\documentclass[reqno]{amsart}
\usepackage{amssymb, amsmath, amsthm, amscd, mathrsfs}

\usepackage{paralist}
\usepackage{cite}
\usepackage{bigints}
\usepackage{dsfont}
\usepackage{empheq}
\usepackage{amssymb}
\usepackage{cases}
\usepackage{enumitem}
\allowdisplaybreaks

\usepackage{hyperref}

\theoremstyle{plain}
\newtheorem{theorem}{Theorem}[section]

\newtheorem{lemma}[theorem]{Lemma}
\newtheorem{proposition}{Proposition}
\theoremstyle{definition}
\newtheorem{definition}[theorem]{Definition}
\newtheorem{remark}{Remark}
\newtheorem{example}{Example}

\def \d {\mathrm{d}}

\title[Stability for the Ornstein-Uhlenbeck equation] 
      {Stability estimates for initial data in general Ornstein-Uhlenbeck equations}

\author{S. E. Chorfi}
\author{L. Maniar}

\address{S. E. Chorfi, L. Maniar, Cadi Ayyad University, Faculty of Sciences Semlalia, LMDP, UMMISCO (IRD-UPMC), B.P. 2390, Marrakesh, Morocco}

\email{s.chorfi@uca.ac.ma, maniar@uca.ac.ma}

\makeatletter 
\@namedef{subjclassname@2020}{%
  \textup{2020} Mathematics Subject Classification}
\makeatother


\begin{document}
\begin{abstract}
We consider the inverse problem of determining initial data in general Ornstein-Uhlenbeck equations on the Euclidean space from partial measurement localized on the so-called thick sets. Using the logarithmic convexity technique and recent observability results, we prove new stability estimates of logarithmic rate for large classes of initial data. Such stability estimates are crucial when dealing with the numerical reconstruction of initial data. Our analysis covers both cases: the analytic Ornstein-Uhlenbeck semigroup on $L^2\left(\mathbb R^N, \mathrm{d}\mu\right)$ with invariant measure $\mu$, and the non-analytic Ornstein-Uhlenbeck semigroup on $L^2\left(\mathbb R^N, \mathrm{d} x\right)$ with the Lebesgue measure. We treat general equations with a diffusion matrix and the Ornstein-Uhlenbeck equation with fractional diffusion in the latter case. This allows us to extend some recent results and simplify some parts of the proof.
\end{abstract}

\dedicatory{\large Dedicated to the memory of Professor Hammadi Bouslous}
\maketitle

\section{Introduction}
Let $N\ge 1$ be an integer and $T >0$ be a fixed terminal time. We consider the following Ornstein-Uhlenbeck equation
\begin{empheq}[left =\empheqlbrace]{alignat=2}
\begin{aligned}\label{e1}
&\partial_t u = \mathrm{tr}\left(Q\nabla^2 u\right) + B x\cdot \nabla u, \qquad 0<t<T , && x\in \mathbb{R}^N, \\
& u\rvert_{t=0}=u_0(x),   && x \in \mathbb{R}^N,
\end{aligned}
\end{empheq}
where $Q=\left(q_{i j}\right)_{1\le i,j\le N}$ is a real symmetric matrix that is positive definite, the drift $0\neq B=\left(b_{i j}\right)_{1\le i,j\le N}$ is a real matrix, while $u_0 \in L^2\left(\mathbb{R}^N\right)$ is an unknown initial datum. The dot notation $``\cdot"$ stands for the standard scalar product in $\mathbb{R}^N$ (also denoted by $\left\langle\cdot, \cdot\right\rangle$). More precisely, we have
\begin{align*}
\mathrm{tr}\left(Q\nabla^2 \cdot\right)&=\sum_{i,j=1}^N q_{ij} \partial_{x_i, x_j}^2,\\
B x \cdot \nabla \cdot &=\sum_{i, j=1}^N b_{i j} x_j \partial_{x_i}.
\end{align*}
The Ornstein-Uhlenbeck equation \eqref{e1} has been mainly studied in two Lebesgue spaces:
\begin{itemize}
    \item $L^2$-space with the invariant measure $L^2_\mu\left(\mathbb{R}^N\right)$: in this case, the Ornstein-Uhlenbeck $C_0$-semigroup is analytic of angle less than $\frac{\pi}{2}$. The existence of an invariant measure $\mu$ holds if and only if $\sigma(B)\subset \{z \in \mathbb{C} \colon \mathrm{Re}\, z <0\}$;
    \item $L^2$-space with the Lebesgue measure $L^2\left(\mathbb{R}^N\right)$: in this setting, the Ornstein-Uhlenbeck operator generates a $C_0$-semigroup that is non-analytic.
\end{itemize}
Henceforth, we denote by $L^2$ either the space $L^2_\mu\left(\mathbb{R}^N\right)$ or the space $L^2\left(\mathbb{R}^N\right)$.

The focus of this paper is on the inverse initial data problem, which involves determining unknown initial data $u_0$ that belong to an admissible set $\mathcal{I}$. This determination is achieved through the following measurement
$$u\rvert_{(0,T) \times \omega},$$
where $\omega \subset \mathbb{R}^N$ is a suitable set of measurements that will be made precise later. Due to the ill-posed nature of such an inverse problem, we mainly aim to address the following question:
\begin{align*}
    \textbf{Stability:} &\text{ is it possible to estimate } \|u_0\|_{L^2} \text{ by a suitable norm}\\
    &\text{ of the quantity } u\rvert_{\left(0, T\right)\times \omega}?
\end{align*}
Within the general parabolic setting, a logarithmic rate is expected for the stability estimate, that is, there exist constants $c>0$ and $\alpha \in (0,1]$ such that
\begin{equation} \label{lsi}
\|u_0\|_{L^2} \le \frac{c}{|\log \|u\rvert_\omega\||^\alpha}
\end{equation}
for all $u_0\in \mathcal{I}$, and for a norm of $u\rvert_\omega$ suitably chosen, e.g., $\|u\|_{L^2\left(0, T ; L^{2}(\omega)\right)}$ or $\|u\|_{H^1\left(0, T ; L^{2}(\omega)\right)}$.

The above estimate \eqref{lsi} is known as conditional stability, which depends on the admissible set $\mathcal{I}$. It should be noted that the rate of conditional stability varies depending on the chosen admissible set. Additionally, conditional stability is particularly relevant when working on the numerical reconstruction of initial data, as demonstrated in several previous works such as \cite{LYZ'09, YZ'08}.

The Ornstein-Uhlenbeck equation is widely studied in view of its various applications, for instance, in stochastic processes, finance, fluid dynamics, and control theory, among other fields. More details can be found in various references \cite{BP'18, CFMP'05, LRM'16, Lu'97, MPP'02, MPRS'03}, but the list is not exhaustive.

Few studies have been conducted for inverse problems of evolution equations with unbounded coefficients such as the Ornstein-Uhlenbeck equation. For instance, Lorenzi \cite{Lo'11} studied the existence and uniqueness of a constant in the drift term. An abstract stability estimate has recently been obtained in  \cite{ACM'21} for general parabolic equations with an arbitrary analyticity angle. The result has been applied to the equation \eqref{e1} with $Q=I_N$ in $L^2_\mu\left(\mathbb{R}^N\right)$, where $I_N$ stands for the identity matrix. The proof is indirect and relies on a sharp lower bound of a harmonic function that appears in the logarithmic convexity estimate, see \cite{AN'63}, \cite{KP'60}, and \cite{Pa'75}. More recently in \cite{CM'23}, the authors have studied the non-analytic case in $L^2\left(\mathbb{R}^N\right)$. In the present paper, we give a direct proof for the analytic case and extend the results in both cases for a general constant matrix $Q$ by considering the Ornstein-Uhlenbeck equation with fractional diffusion in $L^2\left(\mathbb{R}^N\right)$. Finally, it is worth mentioning that logarithmic convexity has been recently extended to time-fractional evolution equations \cite{CMY'22}.

The rest of the paper will be structured as follows: in Section \ref{sec2}, we recall some preliminaries needed throughout the paper. Section \ref{sec3} is devoted to a stability estimate in the space $L^2_\mu\left(\mathbb{R}^N\right)$, where we use logarithmic convexity for analytic semigroups. In particular, we improve some recent results in terms of measurement sets. Finally, we conclude with Section \ref{sec4}, in which we prove a stability estimate for the Ornstein-Uhlenbeck equation with fractional diffusion in the space $L^2\left(\mathbb{R}^N\right)$. Here, we prove a new logarithmic convexity estimate, which holds for this fractional non-analytic case.

\section{Preliminaries}\label{sec2}
We start by gathering preliminary results on the Ornstein-Uhlenbeck equation and beyond, which we will use in the subsequent sections. We refer to the papers \cite{LMP'20, Me'01} and their bibliography for more details.

\subsection{The Ornstein-Uhlenbeck semigroup}
The Ornstein-Uhlenbeck semigroup associated with \eqref{e1} admits an explicit representation formula (also known as Kolmogorov's formula):
\begin{equation}\label{kf}
(\mathcal{T}_0(t) f)(x)=\frac{1}{(4 \pi)^{\frac{N}{2}} (\operatorname{det} Q_{t})^{\frac{1}{2}}} \int_{\mathbb{R}^{N}} \mathrm{e}^{-\frac{1}{4}\left\langle Q_{t}^{-1} y, y\right\rangle} f\left(\mathrm{e}^{t B} x-y\right)\, \d y, \qquad t> 0,
\end{equation}
where $$Q_t=\int_0^t \mathrm{e}^{\tau B}Q\,\mathrm{e}^{\tau B^\top}\, \d \tau,$$
with $B^\top$ denotes the transpose matrix of $B$. The formula \eqref{kf} holds for a large class of functions, e.g., for bounded continuous functions on $\mathbb R^N$.

\subsection*{In the space \texorpdfstring{$L^2_\mu\left(\mathbb{R}^N\right)$}{}}\label{subsec21}
The Ornstein-Uhlenbeck semigroup $(\mathcal{T}_\mu(t))_{t\ge 0}$, on the weighted space $L^2_\mu\left(\mathbb{R}^N\right):=L^2(\mathbb{R}^N,\d \mu)$, enjoys remarkable properties, where $\mu$ denotes the invariant measure i.e.,
$$\int_{\mathbb{R}^{N}} \mathcal{T}_\mu(t) f\, \d \mu=\int_{\mathbb{R}^{N}} f \,\d \mu$$
for all $t\ge 0$ and all $f\in C_b(\mathbb R^N)$ (the space of bounded continuous functions). The existence of $\mu$ is characterized by the following spectral condition
\begin{equation} \label{spec}
\sigma(B) \subset \{z\in \mathbb{C} \colon \mathrm{Re}\,z <0\}.
\end{equation}
Under this condition, the (unique) invariant measure $\mu$ is given by
$$\d \mu(x)=\frac{1}{(4 \pi)^{\frac{N}{2}} (\operatorname{det} Q_{\infty})^{\frac{1}{2}}} \mathrm{e}^{-\frac{1}{4}\left\langle Q_{\infty}^{-1} x, x\right\rangle}\, \d x =:\rho(x)\, \d x,$$
where $$Q_\infty:=\int_0^{\infty} \mathrm{e}^{\tau B} Q\,\mathrm{e}^{\tau B^\top}\, \d \tau.$$
The infinitesimal generator of $(\mathcal{T}_\mu(t))_{t\ge 0}$ is given by
\begin{equation}
\begin{aligned}
\mathcal{A}_\mu &:=\mathrm{tr}\left(Q\nabla^2 \cdot\right) + B x\cdot \nabla \cdot,\\
D(\mathcal{A}_\mu) & =H_{\mu}^{2}:=\left\{u \in L_{\mu}^{2}\left(\mathbb{R}^N\right): \partial^{\alpha} u \in L_{\mu}^{2}\left(\mathbb{R}^N\right) \text { for }|\alpha| \le 2\right\}.
\end{aligned}
\end{equation}
We define the fractional spaces $H_{\mu}^{s}$ for $s>0$, as follows
\begin{equation}
\begin{aligned}\label{fdom}
& H_{\mu}^{s} :=\left\{f \in L_{\mu}^{2}\left(\mathbb{R}^N\right): x \mapsto f(x) \mathrm{e}^{-\frac{1}{8}\left\langle Q_{\infty}^{-1} x, x\right\rangle} \in H^{s}\left(\mathbb{R}^{N}\right)\right\}, \\
&\|f\|_{H_{\mu}^{s}} :=\left\|f \mathrm{e}^{-\frac{1}{8}\left\langle Q_{\infty}^{2} x, x\right\rangle}\right\|_{H^{s}\left(\mathbb{R}^{N}\right)}. 
\end{aligned}
\end{equation}
The operator $-\mathcal{A}_\mu$ is nonnegative on $L^2_\mu\left(\mathbb{R}^N\right)$, and the domain of the fractional powers $(-\mathcal{A}_\mu)^\varepsilon$ is given by
\begin{align}
D((-\mathcal{A}_\mu)^\varepsilon)=H^{2\varepsilon}_\mu, \qquad \varepsilon \in (0,1),
\end{align}
with equivalent norms.

The next theorem states the analyticity of $(\mathcal{T}_\mu(t))_{t\ge 0}$ and gives the precise value of the analyticity angle.
\begin{theorem}[see \cite{CFMP'05}]
The semigroup $(\mathcal{T}_\mu(t))_{t\ge 0}$ is contractive and analytic in the sector of optimal angle $\psi \in \left(0,\dfrac{\pi}{2}\right]$ given by
$$\cot \psi=2\left\|\frac{1}{2} I_N+Q^{-\frac{1}{2}} Q_{\infty} B^\top Q^{-\frac{1}{2}}\right\|.$$
\end{theorem}

\subsection{In the space \texorpdfstring{$L^2\left(\mathbb{R}^N\right)$}{}}
In this case, we consider a more general setting. Let $s>0$ be a positive real number. With the same notations in \eqref{e1}, we consider the Ornstein-Uhlenbeck equation with fractional diffusion:
\begin{empheq}[left =\empheqlbrace]{alignat=2}
\begin{aligned}\label{es}
&\partial_t u = -\mathrm{tr}^s\left(-Q\nabla^2 u\right) + B x\cdot \nabla u, \qquad 0<t<T , && x\in \mathbb{R}^N, \\
& u\rvert_{t=0}=u_0(x),   && x \in \mathbb{R}^N,
\end{aligned}
\end{empheq}
where $\mathrm{tr}^s\left(-Q\nabla^2 \cdot\right)$ is the Fourier multiplier whose symbol is $\left\langle Q \xi, \xi\right\rangle^s$. This fractional model has been recently studied in \cite{Al'20, AB'20}, to which we refer for the details. We introduce the fractional Ornstein-Uhlenbeck operator
\begin{equation}
\begin{aligned}
\mathcal{A} &:=-\mathrm{tr}^s\left(-Q\nabla^2 \cdot\right) + B x\cdot \nabla \cdot,\\
D(\mathcal{A}) & =\left\{u \in L^{2}\left(\mathbb{R}^{N}\right) \colon \mathcal{A}u \in L^{2}\left(\mathbb{R}^{N}\right)\right\}.
\end{aligned}
\end{equation}
The operator $\mathcal{A}$ generates a $C_0$-semigroup $(\mathcal{T}(t))_{t\ge 0}$ on $L^2(\mathbb{R}^N)$, but, in contrast to Subsection \ref{subsec21}, the analyticity property is missing, e.g., for $s=1$, see e.g., \cite{Me'01}. Moreover, by the Fourier transform, the semigroup is given by
\begin{equation}\label{eouf}
\widehat{\mathcal{T}(t) u_0}(\xi)=\mathrm{e}^{-\mathrm{tr}(B) t} \exp \left[- \int_0^t\left|Q^{\frac{1}{2}} \mathrm{e}^{-\tau B^\top} \xi\right|^{2 s} \d \tau\right] \widehat{u_0}\left(\mathrm{e}^{-t B^\top} \xi\right),
\end{equation}
where $\widehat{g}$ denotes the Fourier transform of a function $g$:
$$
\widehat{g}(\xi)=\int_{\mathbb R^N} g(x)~\mathrm{e}^{-\mathrm{i} x \cdot \xi}\, \d x.
$$
By a change of variable, we obtain
\begin{equation}\label{sgn}
    \left\|\widehat{\mathcal{T}(t) u_0}\right\|_{L^2\left(\mathbb{R}^N\right)}=\mathrm{e}^{-\frac{\mathrm{tr}(B)}{2}t}\left\|\exp \left[-\int_0^t\left|Q^{\frac{1}{2}} \mathrm{e}^{\tau B^\top} \cdot\right|^{2 s} \d \tau\right] \widehat{u_0}\right\|_{L^2\left(\mathbb{R}^N\right)}.
\end{equation}

\section{Stability estimate in the space \texorpdfstring{$L^2_\mu\left(\mathbb{R}^N\right)$}{}} \label{sec3}
Let $\varepsilon \in (0,1)$ and $M>0$ be fixed. We introduce the admissible set of initial data
\begin{align}
\mathcal{I}_{\varepsilon,M}:=\{u_0\in H^{2\varepsilon}_\mu \colon \|u_0\|_{H^{2\varepsilon}_\mu} \le M\}, \label{init1}
\end{align}
which is a bounded subset of $D((-\mathcal{A}_\mu)^\varepsilon)$. We aim at proving stability for initial data $u_0\in \mathcal{I}_{\varepsilon, M}$ from the partial measurement $u\rvert_\omega$. The general plan of the proof is as follows:
$$\text{(Logarithmic convexity + Observability inequality)} \Longrightarrow \text{Logarithmic stability}.$$

First, the semigroup $(\mathcal{T}_\mu(t))_{t\ge 0}$ is analytic and contractive on $L^2_\mu\left(\mathbb R^N\right)$ of angle $\psi$. Then the following logarithmic convexity estimate is valid.
\begin{lemma}\label{lemlc}
There exists a constant $K_T>0$ such that, for every $u_0\in L^2_\mu\left(\mathbb R^N\right)$, the following inequality holds
\begin{equation}\label{eqlc2}
\|u_\mu(t)\|_{L^2_\mu} \leq K_T \|u_0\|_{L^2_\mu}^{1- (\frac{t}{rT})^{\phi}} \|u_\mu(T)\|_{L^2_\mu}^{(\frac{t}{rT})^{\phi}}, \qquad 0\leq t\leq T,
\end{equation}
where $u_\mu(t)=\mathcal{T}_\mu(t) u_0$ is the associated solution to \eqref{e1}, with $$r=\begin{cases}
1 &\mbox{ if } \psi=\frac{\pi}{2},\\
2 \cos{(\frac{\psi}{2})} &\mbox{ if } \psi<\frac{\pi}{2}, \end{cases} \qquad \qquad \phi=\begin{cases}
1 &\mbox{ if } \psi=\frac{\pi}{2},\\
\frac{\pi}{\psi} &\mbox{ if } \psi<\frac{\pi}{2}.
\end{cases}$$
\end{lemma}

\begin{proof}
The case $\psi=\frac{\pi}{2}$ follows from \cite[Corollary 3.3]{ACM'22}.
The case $\psi<\frac{\pi}{2}$ is contained in \cite[Theorem 3]{KP'60}.
\end{proof}

Next, we introduce a geometric condition provided for the observability inequality. We refer to \cite{EV'18}, \cite{WWZZ'19}, and the references therein for more details.
\begin{definition}
Let $\lambda \in (0,1]$ and $a=\left(a_{1}, \ldots, a_{N}\right) \in \mathbb{R}^{N}$ such that $a_i>0$, $i=1,2,\ldots,N$. Let us consider the rectangle $\mathcal{R}=\left[0, a_{1}\right] \times \ldots \times\left[0, a_{N}\right]$.
\begin{itemize}
\item[(i)] A measurable subset $\omega \subset \mathbb{R}^{N}$ is said to be $(\lambda, a)$-thick if
$$
\left|\omega \cap \left(x+\mathcal{R}\right)\right| \geq \lambda \prod_{j=1}^{N} a_{j} \qquad \forall x \in \mathbb{R}^{N},  
$$
with $\left|E\right|$ is the Lebesgue measure of a measurable subset $E \subset \mathbb{R}^{N}$.
\item[(ii)] A measurable subset $\omega \subset \mathbb{R}^{N}$ is thick if there exist $\lambda \in (0,1]$ and $a \in \mathbb{R}^{N}$ (as above) such that $\omega$ is $(\lambda, a)$-thick.
\item[(iii)] A measurable set $\omega \subset \mathbb{R}^{N}$ is $\lambda$-thick at scale $L>0$ if $\omega$ is $(\lambda, a)$-thick with $a=(L, \ldots, L) \in \mathbb{R}^{N}$.
\end{itemize}
\end{definition}

\begin{example} Here are some basic examples:
\begin{itemize}
    \item[$\bullet$] In $\mathbb R$, the set $\omega=\bigcup\limits_{n \in \mathbb{Z}}\left[n, n+\frac{1}{2}\right]$ is $\frac{1}{2}$-thick at scale $1$.
    \item[$\bullet$] In $\mathbb{R}^2$, a set given by a periodic arrangement of cubes is thick.
\end{itemize}
\end{example}

The thickness of a subset $\omega$ gives a sufficient condition for the final state observability (see \cite{TW'09}) for equations in unbounded domains such as \eqref{e1} (it is also a necessary condition for the heat equation on $\mathbb R^N$).

Next, we recall the final state observability for system \eqref{e1} in $L^2_\mu\left(\mathbb R^N\right)$ from thick sets. This is a direct consequence of \cite[Theorem 5]{BEP'20} in the autonomous case and the proof of \cite[Corollary 1.7]{BP'18}.
\begin{lemma}
Let $\omega \subset \mathbb{R}^N$ be a thick set. If the spectral condition \eqref{spec} holds, then for all $T>0$, there exists a positive constant $\kappa_T=\kappa_T(\omega,T)$ such that for all $u_0 \in L^2_\mu$, we have
\begin{equation}\label{obsineq}
\|u_\mu(T,\cdot)\|_{L^2_\mu}^2 \leq \kappa_T^2 \int_0^T \|u_\mu(t,\cdot)\|_{L^2_\mu(\omega)}^2\,\d t,
\end{equation}
where $u_\mu$ is the associated solution with \eqref{e1}, with $L^2_\mu(\omega):=L^2(\omega,\d \mu)$. 
\end{lemma}

Let us introduce some notations. First, we set
\begin{align}\label{csts}
    c_\psi:=\left(\frac{1}{r}\right)^{\phi}.
\end{align}
The incomplete Gamma function is defined by
\begin{equation*}
\Gamma(a,x):=\int_x^{\infty} t^{a-1} \mathrm{e}^{-t}\, \d t, \qquad a>0, \quad x \ge 0.
\end{equation*}
In particular, the Gamma function is given by
\begin{equation*}
\Gamma(a)=\Gamma(a,0):=\int_0^{\infty} t^{a-1} \mathrm{e}^{-t}\, \d t, \qquad a>0.
\end{equation*}

Combining the logarithmic convexity \eqref{eqlc2} and the observability inequality \eqref{obsineq}, we obtain the main result of this section.
\begin{theorem}\label{thm4.4}
Let $\omega \subset \mathbb{R}^N$ be a thick set. If the spectral condition \eqref{spec} is satisfied, then there exist positive constants $K(M,T,\varepsilon,\kappa_T)$ and $\alpha=\alpha(\varepsilon)\in (0,1)$ such that, for all $u_0 \in H_\mu^{2\varepsilon}$ with $\left\|u_0\right\|_{H_\mu^{2\varepsilon}} \le M$, one has
\begin{equation}
\|u_0\|_{L^2_\mu} \le K \left(\frac{\Gamma\left(\frac{1}{\phi}\right)}{\left(-c_\psi p\log \|u_\mu\|_{L^2\left(0,T ; L^2_\mu(\omega)\right)}\right)^{\frac{1}{\phi}}\phi}\right)^{\alpha},
\end{equation}
provided that $\|u_\mu\|_{L^2\left(0,T ; L^2_\mu(\omega)\right)}$ is sufficiently small.
\end{theorem}

\begin{proof}
We will modify the proof of \cite[Theorem 3.8]{ACM'22}. By inequality \eqref{eqlc2}, we have
\begin{equation}\label{lcc}
    \|u_\mu(t)\|_{L^2_\mu} \le \kappa_M \left(\frac{\|u_\mu(T)\|_{L^2_\mu}}{\kappa_M}\right)^{c_\psi (\frac{t}{T})^{\phi}},
\end{equation}
where $\kappa_M>0$ is a constant dependent on $M$ ($K$ will be a generic constant). Let $p>1$ and set $E=\left(\dfrac{\|u_\mu(T)\|_{L^2_\mu}}{\kappa_M}\right)^p$. Without loss of generality, we may assume $0<E<1$. By virtue of \eqref{lcc} and a change of variable, we deduce
\begin{align*} \int_0^{T} \|u_\mu(t)\|_{L^2_\mu}^{p} \,\d t & \le \kappa_M^p \int_0^{T} E^{c_\psi (\frac{t}{T})^{\phi}}\, \d t\\
&= \kappa_M^p T \int_0^1 E^{c_\psi t^{\phi}}\, \d t\\
&= \kappa_M^p T \frac{\Gamma(\frac{1}{\phi})-\Gamma(\frac{1}{\phi}, -c_\psi\log E)}{(-c_\psi\log E)^{\frac{1}{\phi}}\phi}.
\end{align*}
Therefore,
\begin{align}\label{E1}
\|u_\mu\|_{L^p\left(0, T ; L^2_\mu\right)} \le \kappa_M T^{\frac{1}{p}}\left(\frac{\Gamma(\frac{1}{\phi})-\Gamma(\frac{1}{\phi}, -c_\psi\log E)}{(-c_\psi\log E)^{\frac{1}{\phi}}\phi}\right)^{\frac{1}{p}}.
\end{align}
By the fractional powers properties, we have
$$
u_\mu'(t)=-(-\mathcal{A}_\mu)^{1-\varepsilon} \mathrm{e}^{t \mathcal{A}_\mu}(-\mathcal{A}_\mu)^{\varepsilon} u_{0},
$$
and the analyticity of the semigroup yields
$$
\|u_\mu'(t)\|_{L^2_\mu} \le \frac{M_\varepsilon}{t^{1-\varepsilon}} \|u_0\|_{D\left((-\mathcal{A}_\mu)^{\varepsilon}\right)}.
$$
After integration we obtain
$$
\|u_\mu'\|_{L^p\left(0, T; L^2_\mu\right)} \le M \frac{M_\varepsilon T^{\frac{1}{p}-(1-\varepsilon)}}{(1-p(1-\varepsilon))^{\frac{1}{p}}}.
$$
On the other hand, we have $\|u_\mu(t)\|_{L^2_\mu} \le \|u_0\|_{L^2_\mu}$ for all $t\in [0,T]$. Then,
$$\|u_\mu\|_{L^p\left(0, T ; L^2_\mu\right)} \le \kappa_M T^{\frac{1}{p}}.$$
Hence, for $p \in \left(1,\dfrac{1}{1-\varepsilon}\right)$, we derive
\begin{align}
\|u_\mu\|_{W^{1, p}\left(0, T ; L^2_\mu\right)} \le \kappa_M T^{\frac{1}{p}}\left(1+\frac{M_\varepsilon}{T^{1-\varepsilon}(1-p(1-\varepsilon))^{\frac{1}{p}}}\right). \label{e2}
\end{align}
Let $0<\gamma<1$ and $\alpha:=\frac{\gamma}{p}$. An interpolation inequality for \eqref{E1} and \eqref{e2} yields
\begin{align}
\|u_\mu\|_{W^{1-\gamma, p}\left(0, T ; L^2_\mu\right)} \le C \kappa_M T^{\frac{1}{p}}\left(1+\frac{M_\varepsilon}{T^{1-\varepsilon}(1-p(1-\varepsilon))^{\frac{1}{p}}}\right)^{1-\gamma} \nonumber\\
\times \left(\frac{\Gamma(\frac{1}{\phi})-\Gamma(\frac{1}{\phi}, -c_\psi\log E)}{(-c_\psi\log E)^{\frac{1}{\phi}}\phi}\right)^{\alpha}. \label{intrp}
\end{align}
By virtue of the continuous embedding
$$
W^{1-\gamma, p}\left(0, T ; L^2_\mu\left(\mathbb{R}^N\right)\right) \subset C\left(\left[0, T\right] ; L^2_\mu\left(\mathbb{R}^N\right)\right)
$$
for $\gamma \in \left(0,1-\dfrac{1}{p}\right)$, we obtain
\begin{align*}
\|u_0\|_{L^2_\mu} \leq K\left(\frac{\Gamma(\frac{1}{\phi})-\Gamma(\frac{1}{\phi}, -c_\psi\log E)}{(-c_\psi\log E)^{\frac{1}{\phi}}\phi}\right)^{\alpha}.
\end{align*}
By using the observability inequality in time $T$,
$$\|u_\mu(T)\|_{L^2_\mu} \le \kappa_T \|u_\mu\|_{L^2(0,T;L^2_\mu(\omega))},$$
we deduce that
$$\left\|u_0\right\|_{L^2_\mu} \le K \left(\frac{\Gamma\left(\frac{1}{\phi}\right)-\Gamma\left(\frac{1}{\phi}, -c_\psi p \log \left(\kappa_T\kappa_M^{-1} \|u_\mu\|_{L^2(0,T;L^2_\mu(\omega))}\right)\right)}{\left(-c_\psi p \log \left(\kappa_T\kappa_M^{-1} \|u_\mu\|_{L^2(0,T;L^2_\mu(\omega))}\right)\right)^{\frac{1}{\phi}}\phi}\right)^{\alpha}.$$
Therefore,
$$\left\|u_0\right\|_{L^2_\mu} \le K \left(\frac{\Gamma\left(\frac{1}{\phi}\right)}{\left(-c_\psi p \log \|u_\mu\|_{L^2(0,T;L^2_\mu(\omega))}\right)^{\frac{1}{\phi}}\phi}\right)^{\alpha}.$$
\end{proof}

Some remarks are in order:
\begin{remark}
Theorem \ref{thm4.4} extends and improves \cite[Proposition 4.5]{ACM'22} that holds for the diffusion matrix $Q=I_N$ and for observation sets $\omega$ satisfying a stronger geometrical condition. Here we have relaxed the measurement sensors to thick sets.
\end{remark}

\begin{remark}
Being different from \cite[Theorem 3.8]{ACM'22} that is based on a sharp lower bound for a harmonic function (see e.g., \cite{Mi'75}), we have given here a direct proof based on Lemma \ref{lemlc}. We have obtained the same kind of estimate, but with different constants $\phi$ and $c_\psi$, where in \cite{ACM'22} the constants were 
\begin{align*}
\phi_0=\frac{\pi}{2\psi} \qquad \text{ and } \qquad c_\psi^0=\frac{2}{\pi} \left(\frac{\psi}{\sin \psi}\right)^{\phi_0},    
\end{align*}
which are compatible with the case $\psi=\frac{\pi}{2}$.
\end{remark}

\section{Stability estimate for the fractional equation in \texorpdfstring{$L^2\left(\mathbb{R}^N\right)$}{}} \label{sec4}
In contrast to the previous section, the lack of analyticity limits us to consider the set of admissible initial data:
$$\mathcal{I}_{M}=\left\{u_0 \in L^2\left(\mathbb R^N\right), \, \mathcal{A} u_0 \in L^2\left(\mathbb R^N\right)  \colon \|u_0\|_{D(\mathcal{A})} \leq M\right\},$$
which is a bounded subset of $D(\mathcal{A})$.

Next, we prove a new logarithmic convexity estimate for the fractional Ornstein-Uhlenbeck semigroup $\left(\mathcal{T}(t)\right)_{t\ge 0}$ on $L^2\left(\mathbb{R}^N\right)$, which extends \cite[Proposition 1]{CM'23} for $s=1$ and $Q=I_N$.
\begin{proposition}\label{proplc}
Let $s>0$ and $T>0$. There exists a constant $c=c(s,T)\in (0,1]$ such that, for every $u_0\in L^2\left(\mathbb R^N\right)$, the following inequality holds
\begin{equation}\label{elc0}
\|u(t)\|_{L^{2}\left(\mathbb{R}^{N}\right)} \le \mathrm{e}^{-\frac{\mathrm{tr}(B)}{2}(1-c)t}\|u_0\|_{L^{2}\left(\mathbb{R}^{N}\right)}^{1-c\frac{t}{T}} \|u(T) \|_{L^{2}\left(\mathbb{R}^{N}\right)}^{c\frac{t}{T}}, \qquad t\in [0,T].
\end{equation}
where $u(t)=\mathcal{T}(t)u_0$ is the solution of the fractional equation \eqref{es}.
\end{proposition}

\begin{proof}
By the Plancherel theorem, the estimate \eqref{elc0} is equivalent to the following inequality
\begin{equation*}
\|\widehat{u}(t)\|_{L^{2}\left(\mathbb{R}^{N}\right)} \le \mathrm{e}^{-\frac{\mathrm{tr}(B)}{2}(1-c)t} \|\widehat{u_0}\|_{L^{2}\left(\mathbb{R}^{N}\right)}^{1-c\frac{t}{T}} \|\widehat{u}(T)\|_{L^{2}\left(\mathbb{R}^{N}\right)}^{c\frac{t}{T}}, \qquad t\in [0,T].
\end{equation*}
By the formula \eqref{sgn}, the above inequality is equivalent to
\begin{align}\label{elc1}
&\qquad \left\|\exp\left[-\int_0^t \left|Q^{\frac{1}{2}} \mathrm{e}^{\tau B^\top} \xi\right|^{2 s} \d \tau\right] \widehat{u_0}\right\|_{L^{2}\left(\mathbb{R}^{N}\right)}\\
& \le \|\widehat{u_0}\|_{L^{2}\left(\mathbb{R}^{N}\right)}^{1-c\frac{t}{T}} \left\|\exp\left[-\int_0^T \left|Q^{\frac{1}{2}} \mathrm{e}^{\tau B^\top} \xi\right|^{2 s} \d \tau\right] \widehat{u_0}\right\|_{L^{2}\left(\mathbb{R}^{N}\right)}^{c\frac{t}{T}}, \;\; t\in [0,T]. \notag
\end{align}
We introduce the function
$$\beta(t,\xi)=
\begin{cases}
\displaystyle\frac{1}{t} \int_0^t \left|Q^{\frac{1}{2}} \mathrm{e}^{\tau B^\top} \xi\right|^{2 s} \d \tau &\mbox{ if } t>0,\\
\left|Q^{\frac{1}{2}} \xi\right|^{2 s} &\mbox{ if } t=0.
\end{cases}$$
The function $\beta$ is continuous and positive on the compact set $[0,T]\times \mathbb{S}^{N-1}$, where $\mathbb{S}^{N-1}:=\{\xi \in \mathbb{R}^N: |\xi|=1 \}$ is the $(N-1)$-dimensional unit sphere. Then, there exist positive constants $c_1 \le c_2$ (dependent on $s$ and $T$) such that
$$c_1 t \le \int_0^t \left|Q^{\frac{1}{2}} \mathrm{e}^{\tau B^\top} \xi\right|^{2 s} \d \tau \le c_2 t \qquad \forall t\in [0,T],\quad \forall \xi\in \mathbb{S}^{N-1},$$
and consequently,
$$c_1 t |\xi|^{2s} \le \int_0^t \left|Q^{\frac{1}{2}} \mathrm{e}^{\tau B^\top} \xi\right|^{2 s} \d \tau \le c_2 t |\xi|^{2s} \qquad \forall t\in [0,T],\quad \forall \xi\in \mathbb{R}^N.$$
The above estimate implies
\begin{equation}\label{eqq}
\int_0^t \left|Q^{\frac{1}{2}} \mathrm{e}^{\tau B^\top} \xi\right|^{2 s} \d \tau \ge c_{s,T} \frac{t}{T} \int_0^T \left|Q^{\frac{1}{2}} \mathrm{e}^{\tau B^\top} \xi\right|^{2 s} \d \tau \;\;\;\; \forall t\in [0,T], \;\; \forall \xi\in \mathbb{R}^N,
\end{equation}
where $c:=c_{s,T}=\frac{c_1}{c2} \le 1$. Therefore, for all $t\in [0,T]$ and all $\xi\in \mathbb{R}^N$, we get the pointwise estimate
\begin{align*}
    &\qquad \exp\left[-2\int_0^t \left|Q^{\frac{1}{2}} \mathrm{e}^{\tau B^\top} \xi\right|^{2 s} \d \tau\right] |\widehat{u_0}(\xi)|^2 \\
    & \quad \le |\widehat{u_0}(\xi)|^{2\left(1-c\frac{t}{T}\right)} \left(\exp\left[-2\int_0^T \left|Q^{\frac{1}{2}} \mathrm{e}^{\tau B^\top} \xi\right|^{2 s} \d \tau\right] |\widehat{u_0}(\xi)|^2\right)^{c\frac{t}{T}}.
\end{align*}
We fix $t \in (0,T)$ and apply the Hölder inequality for $p=\frac{T}{T-ct}$ and $q=\frac{T}{ct}$ to the functions 
$$F(\xi)=|\widehat{u_0}(\xi)|^{2\left(1-c\frac{t}{T}\right)},\qquad   G(\xi)=\left(\exp\left[-2\int_0^T \left|Q^{\frac{1}{2}} \mathrm{e}^{\tau B^\top} \xi\right|^{2 s} \d \tau\right] |\widehat{u_0}(\xi)|^2\right)^{c\frac{t}{T}}.$$
We finally obtain \eqref{elc1}. Thus, the proof of \eqref{elc0} is achieved.
\end{proof}

Now, we state the observability inequality of the fractional equation \eqref{es} in $L^{2}\left(\mathbb{R}^{N}\right)$, which holds only for $s>\frac{1}{2}$. For the proof, see Theorem 1.12 in \cite{AB'20}.
\begin{lemma}
We assume that $s>\frac{1}2$ and that $\omega \subset \mathbb{R}^N$ is a thick set. For all $T>0$, there exists a positive constant $\kappa_{s,T}$ such that for all $u_0 \in L^{2}\left(\mathbb{R}^{N}\right)$, we have
\begin{equation}\label{obsineq1}
\|u(T,\cdot)\|_{L^{2}\left(\mathbb{R}^{N}\right)}^2 \leq \kappa_{s,T}^2 \int_0^T \|u(t,\cdot)\|_{L^{2}(\omega)}^2\,\d t,
\end{equation}
where $u$ is the solution associated with \eqref{es}.
\end{lemma}

\begin{remark}
In the case $0<s\le \frac{1}{2}$, the lack of the final state observability is known for specific cases of the fractional heat equation, see, for instance, \cite{Ko'17, Ko'18, Ko'20}.
\end{remark}

The main result of this section reads as follows:
\begin{theorem}\label{thmlogstab1}
We assume that $s>\frac{1}2$ and that $\omega \subset \mathbb{R}^N$ is a thick set. There exist positive constants $C$ and $C_1$ depending on $(s, T, \omega, M)$ such that, for all $u_0 \in \mathcal{I}_R$,
\begin{equation}
\|u_0\|_{L^{2}\left(\mathbb{R}^N\right)} \leq \frac{-C}{\log \left(C_1\|u\|_{H^1\left(0, T ; L^{2}(\omega)\right)}\right)}
\end{equation}
for $\|u\|_{H^1\left(0, T ; L^{2}(\omega)\right)}$ sufficiently small, where $u$ is the solution of system \eqref{e1}.
\end{theorem}

\begin{proof}
Let $C>0$ denote a generic constant that might vary from line to line. Setting $z=u_t$ and applying \eqref{obsineq1} to $z$, we obtain
\begin{equation}\label{in1}
\|z(T, \cdot)\|_{L^{2}\left(\mathbb{R}^N\right)} \leq C\|z\|_{L^{2}\left(0, T ; L^{2}(\omega)\right)}. 
\end{equation}
By applying Proposition \ref{proplc} to $z$, we see that
\begin{equation}\label{in2}
\left\|z(t, \cdot)\right\|_{L^{2}\left(\mathbb{R}^N\right)} \leq K_T M^{1-c\frac{t}{T}}\left\|z(T, \cdot)\right\|_{L^{2}\left(\mathbb{R}^N\right)}^{c\frac{t}{T}}, \qquad 0 \leq t \leq T, 
\end{equation}
for some constant $K_T>0$. Since
$$u(0, \cdot)=-\int_{0}^{T} z(\tau, \cdot) \d \tau + u(T, \cdot),$$
by the logarithmic convexity \eqref{in2} and the observability inequality \eqref{obsineq1}, we obtain
\begin{align*}
\|u(0, \cdot)\|_{L^{2}\left(\mathbb{R}^N\right)} &\leq \int_{0}^{T}\left\|z(\tau, \cdot)\right\|_{L^{2}\left(\mathbb{R}^N\right)} \d \tau + \|u(T, \cdot)\|_{L^{2}\left(\mathbb{R}^N\right)} \\
&\leq C \int_{0}^{\frac{T}{c}}\left\|z(T, \cdot)\right\|_{L^{2}\left(\mathbb{R}^N\right)}^{c\frac{\tau}{T}} \d \tau + C\|u\|_{L^{2}\left(0, T ; L^{2}(\omega)\right)} \\
&\leq C \frac{T}{c} \frac{\|z(T, \cdot)\|_{L^{2}\left(\mathbb{R}^N\right)}-1}{\log \|z(T, \cdot)\|_{L^{2}\left(\mathbb{R}^N\right)}} + C\|u\|_{L^{2}\left(0, T ; L^{2}(\omega)\right)}\\
& \le C \left(\frac{Z-1}{\log Z} + Z \right),
\end{align*}
where we used $T\le \frac{T}{c}$ ($0<c\le 1$) and we set $Z:=\|z(T, \cdot)\|_{L^{2}\left(\mathbb{R}^N\right)} + C\|u\|_{L^{2}\left(0, T ; L^{2}(\omega)\right)}$. Using the inequality \eqref{in1} (the norm $\|u\|_{H^{1}\left(0, T ; L^{2}(\omega)\right)}$ being small enough), we obtain
\begin{equation}
0 < Z \le C' \|u\|_{H^{1}\left(0, T ; L^{2}(\omega)\right)} <1 \label{eqt1}
\end{equation}
for a certain constant $C'>0$. With help of the elementary inequality
$$\dfrac{x -1}{\log x} +x \leq -\dfrac{1+ \mathrm{e}^{-2}}{\log x} \qquad \text{ for } 0<x <1,$$
we obtain
\begin{align*}
\|u_0\|_{L^{2}\left(\mathbb{R}^N\right)} & \leq \frac{-C}{\log \left(C'\|u\|_{H^1\left(0, T ; L^{2}(\omega)\right)}\right)}.
\end{align*}
\end{proof}

\begin{remark}
If we take $s=1$ and $Q=I_N$ (the diffusion matrix) in Theorem \ref{thmlogstab1}, we recover Theorem 4.1 of \cite{CM'23} as a particular case.
\end{remark}

\section{Conclusion and final comments}
We have studied some inverse problems for determining initial data in the Ornstein-Uhlenbeck equation from sensor sets that are thick. Our findings demonstrate that the logarithmic convexity and the observability inequality imply logarithmic stability for admissible sets of initial data. 

We have considered both analytic and non-analytic cases depending on the space $L^2_\mu\left(\mathbb R^N\right)$ with the invariant measure or the space $L^2\left(\mathbb R^N\right)$ with the Lebesgue measure. In the latter case, we have considered a general framework of the Ornstein-Uhlenbeck equation with fractional diffusion. Moreover, we allow for diffusion matrices $Q$ that are constant, symmetric positive definite. Then it would be of much interest to extend our results for the possibly degenerate hypoelliptic operator $-\mathrm{tr}^s\left(-Q\nabla^2 \cdot\right) + B x\cdot \nabla \cdot$, where $Q$ is only positive semidefinite.

Finally, we emphasize that in the analytic case of the invariant measure, we have only considered the integer case $s=1$. More recently, the existence of a unique invariant measure for the fractional case $s\in (0,1)$ has been proven in \cite[Subsection 6.1]{AFP'22}. However, the analyticity of the associated $C_0$-semigroup still needs to be explored in order to extend our stability results in this case.

\newpage
\bibliographystyle{unsrt}

\end{document}